\documentclass[12pt,draft]{amsart}
\usepackage{amsmath,amssymb,amsthm,bm}
\usepackage{graphicx,enumerate}
\usepackage{layout}

\theoremstyle{plain}
\newtheorem{theorem}{Theorem}
\newtheorem{lemma}[theorem]{Lemma}
\newtheorem{proposition}[theorem]{Proposition}
\newtheorem{corollary}[theorem]{Corollary}

\theoremstyle{definition}

\newtheorem*{acknowledgment}{Acknowledgment}
\newtheorem*{isomproblem}{Field isomorphism problem}
\theoremstyle{remark}

\newcommand{\bZ}{\mathbb{Z}}

\newcommand{\bQ}{\mathbb{Q}}

\title[Field isomorphism problem of $X^3+sX+s$ and related Thue equations]
{A note on the field isomorphism problem of $X^3+sX+s$\\
and related cubic Thue equations}
\author{Akinari Hoshi and Katsuya Miyake}

\thanks{This work was partially supported by Grant-in-Aid for Scientific
Research (C) 19540057 of Japan Society for the Promotion of Science and
Rikkyo University Special Fund for Research.}


\begin{document}
\maketitle
{\small
\begin{center}
Department of Mathematics, Rikkyo University,\\
3--34--1 Nishi-Ikebukuro Toshima-ku, Tokyo,
171--8501, Japan. \\E-mail: \texttt{hoshi@rikkyo.ac.jp}
\end{center}

\begin{center}
Department of Mathematics, School of Fundamental Science and Engineering,\\ Waseda University,
3--4--1 Ohkubo Shinjuku-ku, Tokyo, 169--8555, Japan. \\E-mail: \texttt{miyakek@aoni.waseda.jp}
\end{center}
}
\begin{abstract}
We study the field isomorphism problem of cubic generic polynomial $X^3+sX+s$
over the field of rational numbers with the specialization of the parameter $s$
to nonzero rational integers $m$ via primitive solutions to the family of cubic
Thue equations $x^3-2mx^2y-9mxy^2-m(2m+27)y^3=\lambda$ where $\lambda^2$ is a divisor
of $m^3(4m+27)^5$.
\end{abstract}
{\footnotesize
\begin{quote}
\end{quote}
}

\section{Introduction}\label{seintro}

Let $K$ be an arbitrary field and $K(s)$ the rational function field over $K$
with indeterminate $s$.
Let $S_n$ and $C_n$ denote the symmetric group of degree $n$ and the cyclic group of
order $n$ respectively.
We take the cubic polynomial
\[
f_s(X)=X^3+sX+s\in K(s)[X].
\]
The polynomial $f_s(X)$ is known as a {\it generic} polynomial for $S_3$ over $K$.
Namely the Galois group $\mathrm{Gal}_{K(s)} f_s(X)$
of $f_s(X)$ over $K(s)$ is isomorphic to $S_3$
and every Galois extension $L/M$ over an arbitrary field $M\supset K$ with Galois group
$S_3$ can be obtained as $L=\mathrm{Spl}_M f_a(X)$, the splitting field
of $f_a(X)$ over $M$, for some $a\in M$ (cf. \cite[Section 2.1]{JLY02}).

Because the generic polynomial $f_s(X)=X^3+sX+s$ supplies us all Galois extensions
with Galois group $S_3$ over the base field $K$ by specializing the parameter $s$
to elements $a\in K$, it is natural to ask the following problem:
\begin{isomproblem}
For $a,b\in K$, determine whether $\mathrm{Spl}_K f_a(X)$ and $\mathrm{Spl}_K f_b(X)$
are isomorphic over $K$ or not.
\end{isomproblem}
In the cyclic cubic case, Morton \cite{Mor94} gave an explicit answer to
this problem for the generic polynomial $X^3+sX^2-(s+3)X+1$ for $C_3$ over
a field $K$ with char $K\neq 2$ (see also \cite{Cha96} and \cite{HMa}).
In \cite{HM07}, the authors investigated the field isomorphism problem of $f_s(X)$ over
a field $K$ with char $K\neq 3$ and gave the following theorem:
\begin{theorem}[{\cite[Corollary 8]{HM07}}]\label{thS3isom}
Assume that $\mathrm{char}$ $K\neq 3$.
For $a,b\in K\setminus\{0,-27/4\}$ with $a\neq b$,
the splitting fields of $f_a(X)$ and of $f_b(X)$
over $K$ coincide if and only if there exists $u\in K$ such that
\[
b=\frac{a(u^2+9u-3a)^3}{(u^3-2au^2-9au-2a^2-27a)^2}.
\]
\end{theorem}
For the polynomial $f_a(X)$, we assume that $a\neq 0$, $-27/4$ since the discriminant of
$f_a(X)$ is given by $-a^2(4a+27)$.
As a consequence of Theorem \ref{thS3isom}, we see that for an infinite field $K$ and
a fixed element $a\in K$, there exist infinitely many elements $b\in K$ such that
the field overlap $\mathrm{Spl}_K f_b(X)=\mathrm{Spl}_K f_a(X)$ occurs.

For $m\in\bZ$, we define the cubic form $F_m(X,Y)\in\bZ[X,Y]$ by 
\[
F_m(X,Y):=X^3-2mX^2Y-9mXY^2-m(2m+27)Y^3. 
\]
The aim of this paper is to study the field isomorphism problem of $f_s(X)$ over the field
$\bQ$ of rational numbers with the specialization $s\mapsto m\in \bZ$ in more detail
by Theorem \ref{thS3isom} via certain integral solutions $(x,y)\in\bZ^2$ to the family
of cubic Thue equations
\[
F_m(x,y)=\lambda\in\bZ.
\]
The main result of this paper is as follows:
\begin{theorem}\label{thmain}
Let $m\in\bZ\setminus\{0\}$ and $f_m(X)=X^3+mX+m\in \bZ[X]$.
If there exists $n\in\bZ\setminus\{0\}$ with $n\neq m$ such that
the splitting fields of $f_n(X)$ and of $f_m(X)$ over $\mathbb{Q}$ coincide
then there exists a primitive solution $(x,y)\in\mathbb{Z}^2$ with $y>0$ to
\begin{align}
F_m(x,y)=x^3-2mx^2y-9mxy^2-m(2m+27)y^3=\lambda\tag{$*$}
\end{align}
for such a $\lambda \in \bZ$ as $\lambda^2$ is a divisor of $\,m^3(4m+27)^5$.
In particular, the primitive solution $(x,y)\in\bZ^2$ to $(*)$ can be chosen to 
satisfy the relation
\begin{align}
n=m+\frac{m(4m+27)y(x^2+9xy+27y^2+my^2)(x^3-mx^2y-m^2y^3)}{F_m(x,y)^2}.\label{eqnm}
\end{align}
Conversely if there exists a primitive solution $(x,y)\in\mathbb{Z}^2$ to $(*)$ with
a $\lambda \in \bZ$ such that $\lambda^2$ is a divisor of $\,m^3(4m+27)^5$,
then for the rational number $n\in\bQ$ determined by $(\ref{eqnm})$
the splitting fields of $f_n(X)$ and of $f_m(X)$ over $\mathbb{Q}$ coincide,
except for the cases of $n=0$ and of $n=-27/4$.
For a fixed $m\in\bZ\setminus\{0\}$, furthermore, there exist only finitely
many $n\in\bZ$ such that the splitting field of $f_m(X)$ and of $f_n(X)$ over
$\mathbb{Q}$ coincide.
\end{theorem}
We should consider only primitive solutions $(x,y)\in\bZ^2$ (i.e. $\mathrm{gcd}(x,y)=1$)
because $(x,y)$ and $(cx,cy)$ for $c\in\bZ\setminus\{0\}$ give the same $n\in\bZ$
by (\ref{eqnm}).
For $(x,y)\in\bZ^2$ with $y\neq 0$, 
we may assume that $y>0$ because, if $(x,y)$ is a solution to $(*)$
for $\lambda$, then $(-x,-y)$ becomes a solution to $(*)$ for $-\lambda$. 
For $(x,0)\in\bZ^2$, we may assume that $x>0$. 

For this reason, we exclude $(-1,0)\in\bZ^2$ from primitive solutions. 
Then the trivial primitive solution $(1,0)\in\bZ^2\setminus\{(-1,0)\}$ to $(*)$ 
for $\lambda=1$ implies the trivial equality $n=m$ by (\ref{eqnm}).
The converse also holds if $\mathrm{Gal}_\bQ f_m(X)\cong S_3$.
In general, we may obtain the following assertion which complements Theorem \ref{thmain}:
\begin{theorem}\label{thcomp}
For $m,n\in \bZ\setminus\{0\}$, we assume that
the splitting fields of $f_m(X)$ and of $f_n(X)$ over $\mathbb{Q}$ coincide.
Then we have:\\
$(\mathrm{i})$
If $\mathrm{Gal}_\bQ f_m(X)\cong S_3$ then there exists only one primitive solution 
$(x,y)\in\bZ^2\setminus\{(-1,0)\}$ with $y\geq0$ to $(*)$ which satisfies the condition 
$(\ref{eqnm})$ with respect to $m$ and $n$.
In particular, for $m=n$ such primitive solution is $(1,0)\in \bZ^2$;\\
$(\mathrm{ii})$
If $\mathrm{Gal}_\bQ f_m(X)\cong C_3$ then $m=-b^2-b-7$ for some $b\in\bZ$ and
there are exactly three primitive solutions $(x,y)\in\bZ^2\setminus\{(-1,0)\}$ with $y\geq0$
to $(*)$ which satisfy the condition $(\ref{eqnm})$ with respect to $m$ and $n$.
In particular, for $m=n$ such solutions are given by $(1,0)$, $(b-4,1)$, $(-b-5,1)\in \bZ^2$
with $\lambda=1,-(2b+1)^3, (2b+1)^3$ respectively;\\
$(\mathrm{iii})$
If $\mathrm{Gal}_\bQ f_m(X)\cong C_2$ then $m=n=-8$ and the corresponding primitive
solutions $(x,y)\in\bZ^2\setminus\{(-1,0)\}$ with $y\geq0$ to $(*)$ which satisfy the condition
$(\ref{eqnm})$ are given by $(1,0)$, $(-4,1)\in\bZ^2$ with $\lambda=1,-8$ respectively.
\end{theorem}

By Theorem \ref{thcomp}, we get the following corollary:
\begin{corollary}\label{corNM}
For $m\in\bZ\setminus\{0\}$, let $\mathcal{N}$ be the number of all primitive solutions 
$(x,y)\in\bZ^2\setminus\{(-1,0)\}$ with $y\geq 0$ to $(*)$ for $\lambda\in\bZ$ varying under the 
condition that $\lambda^2$ divides $\,m^3(4m+27)^5$.
Then we have
\begin{align*}
\#\{n\in\bZ\,|\,\mathrm{Spl}_\bQ f_n(X)=\mathrm{Spl}_\bQ f_m(X)\}\leq \frac{1}{\mu}\cdot
\mathcal{N}
\end{align*}
with $\mu=1$, $3$, $2$ according to $\mathrm{Gal}_\bQ f_m(X)\cong S_3$, $C_3$, $C_2$, 
respectively.
\end{corollary}

In Section \ref{sepre}, we prepare some lemmas and recall the result in \cite{HMa}
which gives a generalization of Theorem \ref{thS3isom}
(cf. also the survey paper \cite[Theorem 5.1]{HMc}).
In Section \ref{seproof}, we prove Theorem \ref{thmain} and Theorem \ref{thcomp}. 
We note that for fixed $m,n\in\bZ$ with $\mathrm{Spl}_\bQ f_m(X)=\mathrm{Spl}_\bQ f_n(X)$, 
the corresponding primitive solutions to $(*)$ can be obtained by (\ref{eqnm}) explicitly. 
In Sections \ref{seex1} and \ref{seex2}, we give some numerical examples of
Theorem \ref{thmain} via integral solutions to the family $F_m(X,Y)=\lambda$
of cubic Thue equations.
This computation was achieved by using PARI/GP \cite{PARI2} and we checked it
by Mathematica \cite{Wol07}.
For instance, we will give the following numerical example
(cf. Table $2$ in Section \ref{seex1}):
\begin{corollary}\label{cor105}
For $m\in\bZ\setminus\{0\}$ in the range $-10\leq m\leq 5$,
an integer $n\in\bZ\setminus\{0,m\}$ satisfies that
$\mathrm{Spl}_\bQ f_n(X)=\mathrm{Spl}_\bQ f_m(X)$ if and only if $(m, n)$ is one of 
the following $17$ pairs:
\begin{align*}
&(-10,-106480), (-10,-400), (-9,-3087), (-9,-27), (-7,-1588867),\\
&(-7,-189), (-7,-49), (-6,12), (-6,54), (-6,48000), (-5,625),\\
&(-4,128), (-3, 27), (-2,3456), (1, 300763), (2,208974222), (4,3456000).
\end{align*}
\end{corollary}

\section{Preliminaries}\label{sepre}

We take the cubic generic polynomial
\[
f_s(X)=X^3+sX+s\in K(s)[X]
\]
for $S_3$ over the base field $K$.
The discriminant of $f_s(X)$ is $-s^2(4s+27)$.
Hence the quadratic subfield of $\mathrm{Spl}_{K(s)} f_s(X)$ is given by
$K(s)(\sqrt{-4s-27})$ unless char $K=2$. 
We treat the case where $K=\bQ$ and take the specialization
$s\mapsto m\in\bZ\setminus\{0\}$.
Then we first see the following lemmas:
\begin{lemma}
For $m\in\bZ\setminus\{0\}$,
the splitting field $\mathrm{Spl}_\bQ f_m(X)$ is totally real $(resp.$ totally imaginary$)$ 
if $m\leq -7$ $(resp.$ $-6\leq m)$.
\end{lemma}
\begin{lemma}\label{lem8}
For $m\in\bZ\setminus\{0\}$,
the polynomial $f_m(X)$ is irreducible over $\mathbb{Q}$ except for $m=-8$.
In the case of $m=-8$, $f_{-8}(X)$ splits as $f_{-8}(X)=(X+2)(X^2-2X-4)$ over $\bQ$,
and hence $\mathrm{Spl}_\bQ f_{-8}(X)=\bQ(\sqrt{5})$.
\end{lemma}
\begin{proof}
If $f_m(X)$ is reducible over $\mathbb{Q}$ for $m\in\bZ$ then $f_m(X)$ has a linear factor.
Hence there exist $a,b,c\in\bZ$ such that $f_m(X)=X^3+mX+m=(X-a)(X^2+bX+c)$.
By comparing the coefficients, we have
\[
b=a,\quad c=\frac{a^2}{a+1}=a-1+\frac{1}{a+1},\quad
m=-\frac{a^3}{a+1}=-a^2+a-1+\frac{1}{a+1}.
\]
Hence we see $a\in\{-2,0\}$.
According to $a=-2$ and $a=0$, we have $(a,b,c,m)=(-2,-2,-4,-8)$ and $(0,0,0,0)$, respectively.
\end{proof}
\begin{lemma}\label{lemc3}
For $m\in\bZ\setminus\{0\}$, the Galois group $\mathrm{Gal}_\bQ f_m(X)$ of $f_m(X)$ over
$\bQ$ is isomorphic to $C_3$ if and only if there exists $b\in\bZ$ such that $m=-b^2-b-7$.
\end{lemma}
\begin{proof}
If $\mathrm{Gal}_\bQ f_m(X)\cong C_3$ then there exists $a\in\bZ$ such that
$a^2=-(4m+27)$ because the discriminant $\mathrm{disc}(f_m(X))$
of $f_m(X)$ is $-m^2(4m+27)$.
In this case, $m=-(a^2+27)/4\in\bZ$ and hence $a=2b+1$ for some $b\in\bZ$.
Thus we get $m=-b^2-b-7$.
Conversely if $m=-b^2-b-7$ with $b\in\bZ$ then we have
$\mathrm{Gal}_\bQ f_m(X)\cong C_3$ because $\mathrm{disc}(f_m(X))=(2b+1)^2(b^2+b+7)^2$.
\end{proof}
For $m,n\in\bZ$, we define a sextic polynomial $R_{m,n}(X)\in\bZ[X]$ by
\begin{align*}
R_{m,n}(X)&:=m(X^2+9X-3m)^3-n(X^3-2mX^2-9mX-2m^2-27m)^2\\
&\ =(m-n)X^6+m(4n+27)X^5-m(4mn+9m-18n-243)X^4\\
&\quad\ -m(32mn+162m-54n-729)X^3-m^2(8mn-27m+189n+729)X^2\\
&\quad\ -9m^2(4mn-27m+54n)X-m^2(4m^2n+27m^2+108mn+729n).
\end{align*}
The discriminant of $R_{m,n}(X)$ with respect to $X$ is given by
\begin{align}
\mathrm{disc}(R_{m,n}(X))=m^{10}(4m+27)^{15}n^4(4n+27)^3\in\bZ.\label{discRmn}
\end{align}
It follows from Theorem \ref{thS3isom} that for $m,n\in\bZ\setminus\{0\}$ with $m\neq n$,
there exists $u\in\bQ$ such that $R_{m,n}(u)=0$ if and only if
$\mathrm{Spl}_\bQ f_m(X)=\mathrm{Spl}_\bQ f_n(X)$.

We define the {\it decomposition type} DT$(R)$ of the polynomial $R(X)$ of degree $l$ by
the partition of $l$ induced by the degrees of the irreducible factors of $R(X)$ over $\bQ$.
For $m,n\in\bZ\setminus\{0\}$, we put
\begin{align*}
L_m:=\mathrm{Spl}_\bQ f_m(X),\qquad
G_m:=\mathrm{Gal}_\bQ f_m(X),\qquad
G_{m,n}:=\mathrm{Gal}_\bQ \bigl(f_m(X)f_n(X)\bigr).
\end{align*}
In \cite{HMa} and \cite{HMb}, the authors gave the following theorem for a field $K$ 
with char $K\neq 3$ as a generalization of Theorem \ref{thS3isom}
(cf. also \cite[Theorem 7]{HM07} and \cite[Theorem 5.1]{HMc}).
Here we state the theorem only for the case where $K=\bQ$ and $m,n\in\bZ$.
\begin{theorem}\label{thS3}
For $m,n\in \bZ\setminus\{0\}$ with $m\neq n$, we assume that $\# G_m\geq \# G_n$.
Then the Galois group $G_{m,n}=\mathrm{Gal}_\bQ (f_m(X)f_n(X))$ and the intersection field
$L_m\cap L_n=\mathrm{Spl}_\bQ f_m(X)\cap \mathrm{Spl}_\bQ f_n(X)$ are given
by the decomposition type {\rm DT}$(R_{m,n})$ of $R_{m,n}(X)$ as Table $1$ shows.
\end{theorem}
\begin{center}
\vspace*{1mm}
{\rm Table} $1$\vspace*{4mm}\\
\begin{tabular}{|c|c|c|l|l|}\hline
$G_m$& $G_n$ & $G_{m,n}$ & & ${\rm DT}(R_{m,n})$\\ \hline
& & $S_3\times S_3$ & $L_m\cap L_n=\bQ$ & $6$ \\ \cline{3-5}
& $S_3$ & $(C_3\times C_3)\rtimes C_2$ & $[L_m\cap L_n : \bQ]=2$ & $3,3$\\ \cline{3-5}
\raisebox{-1.6ex}[0cm][0cm]{$S_3$} & & $S_3$ & $L_m=L_n$ & $3,2,1$ \\ \cline{2-5}
& $C_3$ & $S_3\times C_3$ & $L_m\cap L_n=\bQ$ & $6$ \\ \cline{2-5}
& \raisebox{-1.6ex}[0cm][0cm]{$C_2$} & $S_3\times C_2$ & $L_m\not\supset L_n$ & $6$\\
\cline{3-5}
& & $S_3$ & $L_m\supset L_n$ & $3,3$ \\ \hline
& \raisebox{-1.6ex}[0cm][0cm]{$C_3$} & $C_3\times C_3$ & $L_m\neq L_n$ & $3,3$ \\ \cline{3-5}
$C_3$ & & $C_3$ & $L_m=L_n$ & $3,1,1,1$\\ \cline{2-5}
& $C_2$ & $C_6$ & $L_m\cap L_n=\bQ$ & $6$ \\ \hline
\raisebox{-1.6ex}[0cm][0cm]{$C_2$} & \raisebox{-1.6ex}[0cm][0cm]{$C_2$} & $C_2\times C_2$
& $L_m\neq L_n$ & $4,2$ \\ \cline{3-5}
& & $C_2$ & $L_m=L_n$ & $2,2,1,1$ \\ \hline
\end{tabular}
\vspace*{5mm}\\
\end{center}

We did not exclude the case $G_{m,n}\cong C_2\times C_2$ from Table $1$ although 
it does not occur for $m,n\in\bZ\setminus\{0\}$ as we see from Lemma \ref{lem8}. 
This is because the table shows all cases for $m,n\in\bQ\setminus\{0,-27/4\}$ with $m\neq n$.

For $u\in\bQ$ with $R_{m,n}(u)=0$,
we write $u=x/y$ as a quotient of relatively prime integers $x,y\in\bZ$ with $y>0$.
Put $z:=1/u=y/x$ unless $u=0$.
We take
\begin{align*}
R_{m,n}'(X)&=X^6\cdot R_{m,n}(1/X)\\
&=m(1+9X-3mX^2)^3-n(1-2mX-9mX^2-2m^2X^3-27mX^3)^2.
\end{align*}
Then $R_{m,n}(u)=0$ if and only if $R_{m,n}'(z)=0$.
We now consider the case of $m=n$.
We easily see that $m=n$ if and only if $R_{m,n}'(0)=0$ because the constant term of $R_{m,n}'(X)$
is $m-n$.
Note that Theorem \ref{thS3} is valid not only for $R_{m,n}(X)$ but also for $R_{m,n}'(X)$
(cf. for example, \cite{HMa}, \cite{HMb}, \cite{HMc} and \cite{HMd}).
For $n=m$, we see
\[
R_{m,m}'(X)=-m(4m+27)X(mX^2+27X^2+9X+1)(m^2X^3+mX-1).
\]
This means that, in the case of $n=m$, the polynomial $R_{m,m}(X)$ becomes quintic,
and the vanishing root corresponds to the point at infinity (i.e. $[x,y]=[1,0]$).
\begin{lemma}\label{lemDTmm}
For $m\in\bZ\setminus\{0\}$,
the decomposition type ${\rm DT}(R_{m,m})$ of $R_{m,m}(X)$ is given by
\begin{align*}
{\rm DT}(R_{m,m})=
\begin{cases}
3,2,\hspace*{12.3mm} \mathrm{if}\quad G_m\cong S_3,\\
3,1,1,\qquad \mathrm{if}\quad G_m\cong C_3,\\
2,2,1,\qquad \mathrm{if}\quad G_m\cong C_2.
\end{cases}
\end{align*}
\end{lemma}
Next we consider the cases of $n=0$ and of $n=-27/4$.
By (\ref{discRmn}), the sextic polynomials $R_{m,0}(X)$ and $R_{m,-27/4}(X)$
have multiple roots.
Indeed we have $R_{m,0}(X)=m(X^2+9X-3m)^3$ and $R_{m,-27/4}(X)=(4m+27)(X^3+9mX+27m)^2/4$.
\begin{lemma}\label{lem3c}
{\rm (i)}\ For $m\in\bZ\setminus\{0\}$, there exists $u\in\bQ$ such that $R_{m,0}(u)=0$
if and only if there exists $c\in\bZ$ such that $m=3c(c+3)$;\\
{\rm (ii)}\ For $m\in\bZ\setminus\{0\}$, there exists $u\in\bQ$ such that
$R_{m,-27/4}(u)=0$ if and only if $m=-8$.
\end{lemma}
\begin{proof}
(i) If there exists $u\in\bQ$ such that $R_{m,0}(u)=m(u^2+9u-3m)^3=0$,
then there exist integers $a,b\in\bZ$ such that $u^2+9u-3m=(u-a)(u-b)$.
In this case, we see $m=b(b+9)/3$ and $a=-b-9$.
Thus, $b=3c$ for some $c\in\bZ$, and hence $m=3c(c+3)$.
Conversely if $m=3c(c+3)$, then we have $R_{3c(c+3),0}(X)=3c(c+3)(X-3c)^3(X+3c+9)^3$.\\
(ii) If there exists $u\in\bQ$ such that $R_{m,-27/4}(u)=(4m+27)(u^3+9mu+27m)^2/4=0$,
then there exist integers $a,b,c\in\bZ$ such that $u^3+9mu+27m=(u-a)(u^2+bu+c)$.
By comparing the coefficients, we obtain
\begin{align*}
b=a,\quad c=\frac{3a^2}{a+3},\quad m=-\frac{a^3}{9(a+3)}.
\end{align*}
Thus $a=3d$ for some $d\in\bZ$, and
\begin{align*}
m=-\frac{d^3}{d+1}=-d^2+d-1+\frac{1}{d+1}.
\end{align*}
Hence we have $d\in\{-2,0\}$.
According to $d=-2$ and $d=0$, we get $(a,b,c,m)=(-6,-6,-36,-8)$ and $(0,0,0,0)$,
respectively.
\end{proof}

As an application, by using the sextic polynomial $R_{m,n}(X)$ for $u=x=0$,
we get the following example of the overlap
$\mathrm{Spl}_\bQ f_m(X)=\mathrm{Spl}_\bQ f_n(X)$ of the splitting fields.
\begin{proposition}
If $(m,n)\in\bZ^2$ is one of the $8$ pairs
\begin{align*}
&(-54,-12), (-18,-108), (-15,-675), (-14,-5292),\\
&(-13,-4563), (-12,-432), (-9,-27), (27,-3),
\end{align*}
then
$\mathrm{Spl}_\bQ f_m(X)=\mathrm{Spl}_\bQ f_n(X)$.
\end{proposition}
\begin{proof}
If $m$ and $n$ satisfy the equation
\begin{align}
4m^2n+27m^2+108mn+729n=0\label{eq4nm},
\end{align}
then $\mathrm{Spl}_\bQ f_m(X)=\mathrm{Spl}_\bQ f_n(X)$ because the constant term of $R_{m,n}(X)$ is $-m^2(4m^2n+27m^2+108mn+729n)$.
By (\ref{eq4nm}), we have
\[
n=-\frac{27m^2}{(2m+27)^2}=-\frac{1}{3}\left(\frac{9m}{2m+27}\right)^2
=-\frac{1}{3}\left(\frac{1}{2}\left(9-\frac{243}{2m+27}\right)\right)^2.
\]
Hence if $n\in\bZ$, then we have $243/(2m+27)\in\bZ$.
Then we finally see $m=-135$, $-54$, $-27$, $-18$, $-15$, $-14$, $-13$,
$-12$, $-9$, $0$, $27$, $108$.
We exclude the cases $m=-135$, $-27$, $0$, $108$ because we obtain
$(m,n)=(-135,-25/3)$, $(-27,-27)$, $(0,0)$, $(108,-16/3)$, respectively, in these cases.
\end{proof}
Next we put
\[
F_m(X,Y):=X^3-2mX^2Y-9mXY^2-m(2m+27)Y^3\in\bZ[X,Y]
\]
as in Theorem \ref{thmain} and study the cubic polynomial
\[
F_m(X,1)=X^3-2mX^2-9mX-m(2m+27)\in\bZ[X].
\]
The discriminant of $F_m(X,1)$ is $-m^2(4m+27)^3$.
\begin{lemma}
For $m\in\bZ\setminus\{0\}$, the splitting fields of $f_m(X)=X^3+mX+m$ and
of $F_m(X,1)=X^3-2mX^2-9mX-m(2m+27)$ over $\bQ$ coincide.
In particular, the polynomial $F_m(X,1)$ is reducible over $\bQ$
only for $m=-8$, and $F_{-8}(X,1)=(X+2)(X^2+14X+44)$.
\end{lemma}
\begin{proof}
The assertion follows directly from the fact that $f_m(X)$ can be transformed to $F_m(Z,1)$
by the following Tschirnhausen transformation:
\begin{align*}
F_m(Z,1)=\mathrm{Resultant}_X \bigl(f_m(X), Z-(2X^2-3X+2m)\bigr).
\end{align*}
The inverse transformation is also given as
\begin{align*}
f_m(X)=\mathrm{Resultant}_Z \Bigl(F_m(Z,1), X-\frac{2Z^2-(4m+9)Z-6m}{4m+27}\Bigr).
\end{align*}
It follows from Lemma \ref{lem8} that for $m\in\bZ\setminus\{0\}$,
the polynomial $F_m(X,1)$ is reducible over $\bQ$ if and only if $m=-8$.
\end{proof}
\section{Proof of Theorem \ref{thmain} and of Theorem \ref{thcomp}}\label{seproof}

We use Theorem \ref{thS3isom} in the case where $K=\bQ$.
Take
\[
F_m(X,Y)=X^3-2mX^2Y-9mXY^2-m(2m+27)Y^3\in\bZ[X,Y].
\]
For $m\in\bZ\setminus\{0\}$, we assume that there exists $n\in\bZ\setminus\{0\}$,
$n\neq m$, such that $\mathrm{Spl}_\bQ f_n(X)=\mathrm{Spl}_\bQ f_m(X)$.
By Theorem \ref{thS3isom}, there exists $u\in\bQ$ such that
\[
n=\frac{m(u^2+9u-3m)^3}{F_m(u,1)^2}
\]
which is equivalent to the condition that $R_{m,n}(u)=m(u^2+9u-3m)^3-n F_m(u,1)^2=0$. 

We write $u=x/y$ as a quotient of relatively prime integers $x,y\in\bZ$ with $y>0$.
Then we have 
\begin{align}
n&=\frac{m(x^2+9xy-3my^2)^3}{F_m(x,y)^2}\label{eqmm}\\
&=m+\frac{m(4m+27)y(x^2+9xy+27y^2+my^2)(x^3-mx^2y-m^2y^3)}{F_m(x,y)^2}\in\bZ.\nonumber
\end{align}
We put
\[
\lambda:=F_m(x,y).
\]
Then $\lambda^2$ divides $m(4m+27)y(x^2+9xy+27y^2+my^2)(x^3-mx^2y-m^2y^3)$.
\begin{proof}[Proof of Theorem \ref{thmain}]
In order to prove Theorem \ref{thmain},
we should show that $\lambda^2$ divides $m^3(4m+27)^5$.
Then it follows by Thue's theorem that for a fixed $m\in\bZ$, there exist only finitely
many $n\in\bZ$ such that $\mathrm{Spl}_\bQ f_m(X)=\mathrm{Spl}_\bQ f_n(X)$ because
$\lambda$ runs over a finite number of integers.

We will use the standard method via resultant and the Sylvester matrix (cf.
\cite[Theorem 6.1]{Lan78}, \cite[Chapter 8, Section 5]{Lan83} and \cite[Section 1.3]{SWP08}).
Put 
\[
g(u):=m(u^2+9u-3m)^3,\qquad h(u):=F_m(u,1)^2.
\]
We take the resultant
\begin{align}
r_m:=\mathrm{Resultant}_u(g(u),h(u))=m^{12}(4m+27)^{18}\in\bZ\label{RR}
\end{align}
of $g(u)$ and $h(u)$ with respect to $u$.

Now we expand the polynomials $g(u)$ and $h(u)$ with respect to $u$ as
\begin{align*}
g(u)=\sum_{i=0}^6 a_i u^{6-i}&=mu^6+27mu^5-9(m-27)mu^4-81m(2m-9)u^3\\
&\quad\,+27(m-27)m^2u^2+243m^3u-27m^4,\\
h(u)=\sum_{i=0}^6 b_i u^{6-i}&=u^6-4mu^5+2m(2m-9)u^4+2m(16m-27)u^3\\
&\quad\,+m^2(8m+189)u^2+18m^2(2m+27)u+m^2(2m+27)^2.
\end{align*}
Then the resultant $r_m$ is also given as the following determinant:
\[
r_m={\small
\left|
\begin{array}{llllllllllll}
 a_0 & a_1 & a_2 & a_3 & a_4 & a_5 & a_6 & 0 & 0 & 0 & 0 & g(u)u^5 \\
 0 & a_0 & a_1 & a_2 & a_3 & a_4 & a_5 & a_6 & 0 & 0 & 0 & g(u)u^4 \\
 0 & 0 & a_0 & a_1 & a_2 & a_3 & a_4 & a_5 & a_6 & 0 & 0 & g(u)u^3 \\
 0 & 0 & 0 & a_0 & a_1 & a_2 & a_3 & a_4 & a_5 & a_6 & 0 & g(u)u^2 \\
 0 & 0 & 0 & 0 & a_0 & a_1 & a_2 & a_3 & a_4 & a_5 & a_6 & g(u)u \\
 0 & 0 & 0 & 0 & 0 & a_0 & a_1 & a_2 & a_3 & a_4 & a_5 & g(u) \\
 b_0 & b_1 & b_2 & b_3 & b_4 & b_5 & b_6 & 0 & 0 & 0 & 0 & h(u)u^5 \\
 0 & b_0 & b_1 & b_2 & b_3 & b_4 & b_5 & b_6 & 0 & 0 & 0 & h(u)u^4 \\
 0 & 0 & b_0 & b_1 & b_2 & b_3 & b_4 & b_5 & b_6 & 0 & 0 & h(u)u^3 \\
 0 & 0 & 0 & b_0 & b_1 & b_2 & b_3 & b_4 & b_5 & b_6 & 0 & h(u)u^2 \\
 0 & 0 & 0 & 0 & b_0 & b_1 & b_2 & b_3 & b_4 & b_5 & b_6 & h(u)u \\
 0 & 0 & 0 & 0 & 0 & b_0 & b_1 & b_2 & b_3 & b_4 & b_5 & h(u)
\end{array}
\right|
};
\]
hence we have
\begin{align}
r_m=m^9(4m+27)^{13}\Bigl(g(u)p(u)+h(u)q(u)\Bigr)\label{RR11}
\end{align}
where
\begin{align*}
p(u)&=-15u^5+3(19m+9)u^4-(48m^2-176m+27)u^3-3m(4m^2+125m-135)u^2\\
&\quad\, -6m(26m^2+339m+162)u-m(28m^3+636m^2+3591m-1458),\\
q(u)&=m\big(15u^5+3(m+126)u^4-(68m-2943)u^3-(26m^2+1809m-5103)u^2\\
&\quad\, +9(5m^2-810m-1458)u+67m^3+2538m^2+11664m+19683\big).
\end{align*}
By (\ref{RR}) and (\ref{RR11}), we have
\begin{align*}
g(u)p(u)+h(u)q(u)=m^3(4m+27)^5.\label{hpfq}
\end{align*}

Put
\begin{align*}
G(x,y):=y^6\cdot g(x/y),\quad P(x,y):=y^5\cdot p(x/y),\quad Q(x,y):=y^5\cdot q(x/y).
\end{align*}
Then
\begin{align*}
G(x,y)P(x,y)+F_m(x,y)^2 Q(x,y)=m^3(4m+27)^5y^{11}.
\end{align*}
Hence we get
\begin{align*}
&\frac{G(x,y)P(x,y)}{\lambda^2}+Q(x,y)=\frac{m^3(4m+27)^5y^{11}}{\lambda^2}\in\bZ.
\end{align*}

It follows from $\mathrm{gcd}(x,y)=1$ and $\lambda=x^3-2mx^2y-9mxy^2-m(2m+27)y^3$ 
that $\mathrm{gcd}(\lambda,y)=1$. 
Hence we conclude that $\lambda^2=F_m(x,y)^2$ divides $m^3(4m+27)^5$. 

Conversely if $n\in\bQ$ is given by (\ref{eqnm}) in Theorem \ref{thmain} then,
by (\ref{eqmm}), the sextic polynomial $R_{m,n}(X)=m(X^2+9X-3m)^3-n F_m(X,1)^2$ has a
root $u\in\bQ$, and hence $\mathrm{Spl}_\bQ f_m(X)=\mathrm{Spl}_\bQ f_n(X)$ except for 
the cases of $n=0$ and of $n=-27/4$.
We note that for $m\in\bZ\setminus\{0\}$
if $m\neq 3c(c+3)$ for any $c\in\bZ$ (resp. $m\neq -8$) then
$n\neq 0$ (resp. $n\neq -27/4$) by Lemma \ref{lem3c}.
\end{proof}
\begin{proof}[Proof of Theorem \ref{thcomp}]

We give a proof of Theorem \ref{thcomp} by using Theorem \ref{thmain}, Theorem \ref{thS3} and
Lemmas \ref{lem8}, \ref{lemc3} and \ref{lemDTmm} in the previous section. 

We first treat the case where $m\neq n$. 
From the assumption, there exists $u\in\bQ$ such that 
$R_{m,n}(u)=m(u^2+9u-3m)^3-n F_m(u,1)^2=0$. 
We write $u=x/y$ with $x,y\in\bZ$, $y>0$ and $\mathrm{gcd}(x,y)=1$. 
Then, by (\ref{eqmm}), $R_{m,n}(u)=0$ if and only if $(x,y)\in\bZ^2$ 
satisfies the condition (\ref{eqnm}). 

We note that the sextic polynomial $R_{m,n}(X)$ has no multiple roots by (\ref{discRmn}). 
Hence it follows from Theorem \ref{thS3} that the number of the set 
$\{u\in\bQ\,|\, R_{m,n}(u)=0\}$ equals exactly $\mu$ where $\mu=1$, $3$, $2$ 
according to $\mathrm{Gal}_\bQ f_m(X)\cong S_3$, $C_3$, $C_2$, respectively 
(see the number of $1$'s in $\mathrm{DT}(R_{m,n})$ as on Table $1$). 

We also see that by the assumption $m\neq n$, $(1,0)\in\bZ^2$ dose not hold the 
condition (\ref{eqnm}). 
Thus the number of the primitive solutions $(x,y)\in\bZ^2\setminus\{(-1,0)\}$ with $y\geq 0$ 
to $(*)$ which satisfies (\ref{eqnm}) 
is exactly $\mu$ and the assertion follows when $m\neq n$. 

In the case where $n=m$, we recall that $R_{m,m}(X)$ becomes quintic and 
$(1,0)\in\bZ^2$ corresponds the vanishing root. 
By Lemma \ref{lemDTmm}, the number of the set $\{u\in\bQ\,|\, R_{m,n}(u)=0\}$ 
equals exactly $\mu-1$ and $(1,0)\in\bZ^2$ satisfies the condition (\ref{eqnm}). 
Hence the number of the primitive solutions $(x,y)\in\bZ^2\setminus\{(-1,0)\}$ with $y\geq 0$ 
to $(*)$ which satisfies (\ref{eqnm}) also equals $\mu$.  

For $m=n$, the corresponding primitive solutions $(x,y)\in\bZ^2\setminus\{(-1,0)\}$ satisfy 
\begin{align*}
\varphi_m(x,y):=y(x^2+9xy+27y^2+my^2)(x^3-mx^2y-m^2y^3)=0.
\end{align*}

If $\mathrm{Gal}_\bQ f_m(X)\cong C_3$ then by Lemma \ref{lemc3} there exists $b\in\bZ$
such that $m=-b^2-b-7$, and hence we have
\begin{align*}
&\varphi_{-b^2-b-7}(x,y)\\
&=y\bigl(x-(b-4)y\bigr)\bigl(x+(b+5)y\bigr)\bigl(x^3+(b^2+b+7)x^2y-(b^2+b+7)^2y^3\bigr).
\end{align*}
Thus, the corresponding three primitive solutions are $(1,0)$, $(b-4,1)$, $(-b-5,1)\in\bZ^2$. 

If $\mathrm{Gal}_\bQ f_m(X)\cong C_2$ then $m=-8$ by Lemma \ref{lem8}.
We see
\[
\varphi_{-8}(x,y)
=y(x+4y)(x^2+9xy+19y^2)(x^2+4xy-16y^2).
\]
Hence the corresponding two primitive solutions are given by $(1,0)$ and $(-4,1)$. 
\end{proof}
\section{Numerical example $1$}\label{seex1}

We give some numerical examples of Theorem \ref{thmain}, Theorem \ref{thcomp} and
Corollary \ref{corNM}.
By using PARI/GP \cite{PARI2}, we computed for $m\in\bZ\setminus\{0\}$
in the range $-10\leq m\leq 5$, all primitive solutions $(x,y)\in\bZ^2\setminus\{(-1,0)\}$
with $y\geq 0$ to $F_m(x,y)=\lambda$ where $\lambda^2$ is a divisor of $m^3(4m+27)^5$
as on the following table (Table $2$).
We checked Table $2$ by Mathematica \cite{Wol07}.
By the result on Table $2$, we get Corollary \ref{cor105} which we stated in Section
\ref{seintro}, because all integers $n\in\bZ$ which satisfy $\mathrm{Spl}_\bQ f_n(X)=
\mathrm{Spl}_\bQ f_m(X)$ appear on Table $2$.

\begin{center}
{\small
\vspace*{1mm}
{\rm Table} $2$\vspace*{4mm}\\
}
{\scriptsize
\begin{tabular}{|c|c|c|c|} \hline
$m$ & $\lambda$ & $F_m(x,y)=\lambda$ & $n$ \\\hline
 & $1$ & $(1,0)$ & $-10$\\\cline{2-4}
 & $-1$ & $(-1,1)$ & $-106480$\\\cline{2-4}
 & $-5$ & $(-5,1)$ & $-400$\\\cline{2-4}
 & $13^2$ & $(-11,1)$ & $-640/13$\\\cline{2-4}
$-10$ & $-13^2$ & $(-9,2)$ & $-270/13$\\\cline{2-4}
 & $2\cdot13^2$ & $(2,1)$ & $-160/13$\\\cline{2-4}
 & $-5\cdot13^2$ & $(-125,9)$ & $-90792400/13$\\\cline{2-4}
 & $-5\cdot13^2$ & $(-5,4)$ & $-6250/13$\\\cline{2-4}
 & $2\cdot5\cdot13^2$ & $(-20,3)$ & $-100/13$\\\hline\hline
 & $1$ & $(1,0)$ & $-9$\\\cline{2-4}
 & $1$ & $(-5,1)$ & $-3087$\\\cline{2-4}
 & $3^3$ & $(-6,1)$ & $-9$\\\cline{2-4}
 & $-3^3$ & $(-3,1)$ & $-9$\\\cline{2-4}
$-9$ & $-3^3$ & $(-12,1)$ & $-3087$\\\cline{2-4}
 & $-3^3$ & $(-3,2)$ & $-3087$\\\cline{2-4}
 & $3^4$ & $(0,1)$ & $-27$\\\cline{2-4}
 & $3^4$ & $(-9,1)$ & $-27$\\\cline{2-4}
 & $-3^4$ & $(-9,2)$ & $-27$\\\hline\hline
 & $1$ & $(1,0)$ & $-8$\\\cline{2-4}
 & $-2^3$ & $(-4,1)$ & $-8$\\\cline{2-4}
 & $2^4$ & $(-6,1)$ & $[\ -27/4\ ]$\\\cline{2-4}
 & $5^2$ & $(-7,1)$ & $-64/5$\\\cline{2-4}
$-8$ & $-5^2$ & $(-9,2)$ & $-216/5$\\\cline{2-4}
 & $2^3\cdot5^2$ & $(-16,3)$ & $-64/5$\\\cline{2-4}
 & $-2^3\cdot5^2$ & $(-12,1)$ & $-216/5$\\\cline{2-4}
 & $2^4\cdot5^2$ & $(-26,3)$ & $-6859/20$\\\cline{2-4}
 & $2^4\cdot5^2$ & $(-34,7)$ & $-6859/20$\\\hline\hline
 & $1$ & $(1,0)$ & $-7$\\\cline{2-4}
 & $1$ & $(-3,1)$ & $-189$\\\cline{2-4}
 & $1$ & $(-5,1)$ & $-7$\\\cline{2-4}
 & $1$ & $(-6,1)$ & $-189$\\\cline{2-4}
 & $1$ & $(-41,9)$ & $-1588867$\\\cline{2-4}
\raisebox{-1.6ex}[0cm][0cm]{$-7$} &
   $-1$ & $(-4,1)$ & $-7$\\\cline{2-4}
 & $-1$ & $(-9,2)$ & $-189$\\\cline{2-4}
 & $-1$ & $(-25,4)$ & $-1588867$\\\cline{2-4}
 & $-1$ & $(-16,5)$ & $-1588867$\\\cline{2-4}
 & $7$ & $(-14,3)$ & $-49$\\\cline{2-4}
 & $-7$ & $(-7,1)$ & $-49$\\\cline{2-4}
 & $-7$ & $(-7,2)$ & $-49$\\\hline\hline
 & $1$ & $(1,0)$ & $-6$\\\cline{2-4}
 & $-1$ & $(-13,3)$ & $48000$\\\cline{2-4}
\raisebox{-1.6ex}[0cm][0cm]{$-6$} & $2$ & $(-4,1)$ & $12$\\\cline{2-4}
 & $3^2$ & $(-3,1)$ & $[\ 0\ ]$\\\cline{2-4}
 & $-3^2$ & $(-9,2)$ & $54$\\\cline{2-4}
 & $-2\cdot3^2$ & $(-6,1)$ & $[\ 0\ ]$\\\hline\hline
 & $1$ & $(1,0)$ & $-5$\\\cline{2-4}
 & $1$ & $(-4,1)$ & $625$\\\cline{2-4}
$-5$ & $7^2$ & $(-1,1)$ & $-5/7$\\\cline{2-4}
 & $-7^2$ & $(-9,2)$ & $135/7$\\\cline{2-4}
 & $5\cdot7^2$ & $(-10,3)$ & $25/7$\\\hline\hline
 & $1$ & $(1,0)$ & $-4$\\\cline{2-4}
$-4$ & $-2^2$ & $(-4,1)$ & $128$\\\cline{2-4}
 & $11^2$ & $(1,1)$ & $-32/11$\\\hline
\end{tabular}
\hspace*{3mm}
\begin{tabular}{|c|c|c|c|} \hline
$m$  & $\lambda$ & $F_m(x,y)=\lambda$ & $n$ \\\hline
 & $-11^2$ & $(-9,2)$ & $108/11$\\\cline{2-4}
 & $2^2\cdot11^2$ & $(-8,3)$ & $16/11$\\\cline{2-4}
$-4$ & $2^2\cdot11^2$ & $(-26,7)$ & $9826/11$\\\cline{2-4}
 & $-2^2\cdot11^2$ & $(-10,1)$ & $-2/11$\\\cline{2-4}
 & $-2^2\cdot11^2$ & $(-1384,365)$ & $206613902738896/11$\\\hline\hline
 & $1$ & $(1,0)$ & $-3$\\\cline{2-4}
 & $3^2$ & $(-3,1)$ & $27$\\\cline{2-4}
\raisebox{-1.6ex}[0cm][0cm]{$-3$} & $5^2$ & $(-2,1)$ & $3/5$\\\cline{2-4}
 & $3^2\cdot5^2$ & $(3,1)$ & $-27/5$\\\cline{2-4}
 & $3^2\cdot5^2$ & $(-24,7)$ & $35937/5$\\\cline{2-4}
 & $-3^2\cdot5^2$ & $(-9,2)$ & $27/5$\\\hline\hline
 & $1$ & $(1,0)$ & $-2$\\\cline{2-4}
 & $1$ & $(-3,1)$ & $3456$\\\cline{2-4}
 & $19^2$ & $(5,1)$ & $-128/19$\\\cline{2-4}
$-2$ & $19^2$ & $(-259,85)$ & $196710433792/19$\\\cline{2-4}
 & $-19^2$ & $(-9,2)$ & $54/19$\\\cline{2-4}
 & $2\cdot19^2$ & $(-4,3)$ & $4/19$\\\cline{2-4}
 & $-2\cdot19^2$ & $(-22,7)$ & $16384/19$\\\hline\hline
 & $1$ & $(1,0)$ & $-1$\\\cline{2-4}
 & $23^2$ & $(7,1)$ & $-125/23$\\\cline{2-4}
\raisebox{-1.6ex}[0cm][0cm]{$-1$} & $23^2$ & $(-2,3)$ & $1/23$\\\cline{2-4}
 & $23^2$ & $(-11,5)$ & $2197/23$\\\cline{2-4}
 & $-23^2$ & $(-9,2)$ & $27/23$\\\cline{2-4}
 & $-23^2$ & $(-42,17)$ & $4492125/23$\\\hline\hline
 & $1$ & $(1,0)$ & $1$\\\cline{2-4}
 & $1$ & $(5,1)$ & $300763$\\\cline{2-4}
\raisebox{-1.6ex}[0cm][0cm]{$1$} & $31^2$ & $(11,1)$ & $343/31$\\\cline{2-4}
 & $-31^2$ & $(-9,2)$ & $-27/31$\\\cline{2-4}
 & $-31^2$ & $(2,3)$ & $1/31$\\\cline{2-4}
 & $-31^2$ & $(24,5)$ & $132651/31$\\\hline\hline
 & $1$ & $(1,0)$ & $2$\\\cline{2-4}
 & $-1$ & $(15,2)$ & $208974222$\\\cline{2-4}
 & $-7^2$ & $(-1,1)$ & $-16/7$\\\cline{2-4}
 & $2\cdot5^2$ & $(8,1)$ & $8788/5$\\\cline{2-4}
$2$ & $-2\cdot5^2$ & $(-2,1)$ & $-32/5$\\\cline{2-4}
 & $-2\cdot7^2$ & $(6,1)$ & $864/7$\\\cline{2-4}
 & $5^2\cdot7^2$ & $(13,1)$ & $1024/35$\\\cline{2-4}
 & $-5^2\cdot7^2$ & $(-9,2)$ & $-54/35$\\\cline{2-4}
 & $-2\cdot5^2\cdot7^2$ & $(4,3)$ & $4/35$\\\hline\hline
 & $1$ & $(1,0)$ & $3$\\\cline{2-4}
 & $13^2$ & $(49,5)$ & $114818259/13$\\\cline{2-4}
$3$ & $-13^2$ & $(2,1)$ & $3/13$\\\cline{2-4}
 & $3^2\cdot13^2$ & $(15,1)$ & $729/13$\\\cline{2-4}
 & $-3^2\cdot13^2$ & $(-9,2)$ & $-27/13$\\\hline\hline
 & $1$ & $(1,0)$ & $4$\\\cline{2-4}
 & $2^2$ & $(12,1)$ & $3456000$\\\cline{2-4}
$4$ & $43^2$ & $(17,1)$ & $4000/43$\\\cline{2-4}
 & $-43^2$ & $(-9,2)$ & $-108/43$\\\cline{2-4}
 & $-2^2\cdot43^2$ & $(8,3)$ & $16/43$\\\hline\hline
 & $1$ & $(1,0)$ & $5$\\\cline{2-4}
\raisebox{-1.6ex}[0cm][0cm]{$5$} & $47^2$ & $(19,1)$ & $6655/47$\\\cline{2-4}
 & $-47^2$ & $(-9,2)$ & $-135/47$\\\cline{2-4}
 & $-5\cdot47^2$ & $(10,3)$ & $25/47$\\\hline
\end{tabular}
}
\end{center}

\section{Numerical example $2$}\label{seex2}

By Theorem \ref{thS3isom}, we evaluated those integers $m,n\in\bZ\setminus\{0\}$ in the range
\begin{align*}
&{\rm (i)} -6\leq m<n\leq 2\times 10^5,\ \mathrm{i.e.}\ \mathrm{Spl}_\bQ f_m(X)
\ \mathrm{is\ totally\ imaginary},\\
&{\rm (ii)} -2\times 10^5\leq n<m\leq -7,\ \mathrm{i.e.}\ \mathrm{Spl}_\bQ f_m(X)
\ \mathrm{is\ totally\ real}
\end{align*}
for which $R_{m,n}(X)$ has a linear factor over $\bQ$, namely,
$\mathrm{Spl}_\bQ f_m(X)=\mathrm{Spl}_\bQ f_n(X)$. 
Note that we should check only for $m,n\in\bZ\setminus\{0\}$ which satisfy that 
$(4m+27)(4n+27)$ is square. 
We also computed the corresponding primitive solutions $(x,y)\in\bZ^2$
to $F_m(x,y)=\lambda$ and to $F_n(x,y)=\lambda'$ with respect to $\{m,n\}$ and to
$\{n,m\}$ respectively.
The result of the case (i) (resp. (ii)) is given on Table $3$ (resp. on Table $4$).

\begin{center}
{\small
\vspace*{3mm}
{\rm Table} $3$\vspace*{4mm}\\
}
{\scriptsize
\begin{tabular}{|c|c|c|c|c|c|}\hline
$m$ & $n$ & $F_m(x,y)=\lambda$ & $\lambda$ & $F_n(x,y)=\lambda'$ & $\lambda'$\\\hline
$-6$ & $12$ & $(-4,1)$ & $2$ & $(-2,1)$ & $-2^2\cdot5^3$\\\hline
$-6$ & $54$ & $(-9,2)$ & $-3^2$ & $(-9,2)$ & $-3^{10}$\\\hline
$-6$ & $48000$ & $(-13,3)$ & $-1$ & $(-140,3)$ & $-2^6\cdot5^3\cdot11^3\cdot23^3$\\\hline
$-5$ & $625$ & $(-4,1)$ & $1$ & $(5,1)$ & $-5^3\cdot19^3$\\\hline
$-4$ & $128$ & $(-4,1)$ & $-2^2$ & $(-8,1)$ & $-2^7\cdot7^3$\\\hline
$-3$ & $27$ & $(-3,1)$ & $3^2$ & $(0,1)$ & $-3^7$\\\hline
$-2$ & $3456$ & $(-3,1)$ & $1$ & $(36,1)$ & $-2^6\cdot3^{12}$\\\hline
$12$ & $54$ & $(18,1)$ & $-2^2\cdot3^2\cdot5^3$ & $(36,1)$ & $-2\cdot3^{10}$\\\hline
$12$ & $48000$ & $(28,1)$ & $-2^2\cdot5^3$ & $(1640,1)$
& $-2^7\cdot5^3\cdot11^3\cdot23^3$\\\hline
$54$ & $48000$ & $(117,1)$ & $3^{10}$ & $(-3420,1)$
& $-2^6\cdot3^2\cdot5^3\cdot11^3\cdot23^3$\\\hline
\end{tabular}
}
\end{center}

\begin{center}
{\small \vspace*{3mm}
{\rm Table} $4$\vspace*{4mm}\\
}

{\scriptsize
\begin{tabular}{|c|c|c|c|c|c|} \hline
$m$ & $n$ & $F_m(x,y)=\lambda$ & $\lambda$ & $F_n(x,y)=\lambda'$ & $\lambda'$\\\hline
$-7$ & $-49$ & $(-7,1)$ & $-7$ & $(-35,2)$ & $7^2\cdot13^3$\\\hline
$-7$ & $-49$ & $(-7,2)$ & $-7$ & $(28,1)$ & $7^2\cdot13^3$\\\hline
$-7$ & $-49$ & $(-14,3)$ & $7$ & $(-7,3)$ & $-7^2\cdot13^3$\\\hline
$-7$ & $-189$ & $(-3,1)$ & $1$ & $(36,1)$ & $3^{12}$\\\hline
$-7$ & $-189$ & $(-6,1)$ & $1$ & $(-45,1)$ & $3^{12}$\\\hline
$-7$ & $-189$ & $(-9,2)$ & $-1$ & $(-9,2)$ & $-3^{12}$\\\hline
$-9$ & $-27$ & $(0,1)$ & $3^4$ & $(9,1)$ & $3^8$\\\hline
$-9$ & $-27$ & $(-9,1)$ & $3^4$ & $(-18,1)$ & $3^8$\\\hline
$-9$ & $-27$ & $(-9,2)$ & $-3^4$ & $(-9,2)$ & $-3^8$\\\hline
$-9$ & $-3087$ & $(-5,1)$ & $1$ & $(14,1)$ & $7^3\cdot37^3$\\\hline
$-9$ & $-3087$ & $(-12,1)$ & $-3^3$ & $(-231,2)$ & $-3^3\cdot7^3\cdot37^3$\\\hline
$-9$ & $-3087$ & $(-3,2)$ & $-3^3$ & $(273,1)$ & $3^3\cdot7^3\cdot37^3$\\\hline
$-10$ & $-400$ & $(-5,1)$ & $-5$ & $(-10,1)$ & $-2^3\cdot5^2\cdot11^3$\\\hline
$-10$ & $-106480$ & $(-1,1)$ & $-1$ & $(-638,1)$ & $2^3\cdot11^3\cdot181^3$\\\hline
$-12$ & $-54$ & $(-6,1)$ & $2^2\cdot3^2$ & $(0,1)$ & $2\cdot3^7$\\\hline
$-12$ & $-432$ & $(0,1)$ & $2^2\cdot3^2$ & $(36,1)$ & $2^4\cdot3^{10}$\\\hline
$-12$ & $-71874$ & $(-18,1)$ & $2^2\cdot3^2$ & $(-1584,1)$
& $2\cdot3^{10}\cdot11^3\cdot13^3$\\\hline
$-13$ & $-4563$ & $(0,1)$ & $13$ & $(117,1)$ & $3^{12}\cdot13^2$\\\hline
$-13$ & $-4563$ & $(-39,2)$ & $5^3\cdot13$ & $(-819,2)$ & $3^{12}\cdot5^3\cdot13^2$\\\hline
$-13$ & $-4563$ & $(-39,7)$ & $-5^3\cdot13$ & $(-234,7)$ & $-3^{12}\cdot5^3\cdot13^2$\\\hline
$-14$ & $-5292$ & $(0,1)$ & $-2\cdot7$ & $(-126,1)$ & $-2^2\cdot3^{12}\cdot7^2$\\\hline
$-15$ & $-675$ & $(0,1)$ & $-3^2\cdot5$ & $(-45,1)$ & $3^{10}\cdot5^2$\\\hline
$-15$ & $-3645$ & $(-6,1)$ & $3^2$ & $(27,1)$ & $-3^{10}\cdot7^3$\\\hline
$-16$ & $-6750$ & $(-6,1)$ & $-2^3$ & $(-45,1)$ & $-3^{12}\cdot5^3$\\\hline
\end{tabular}
}
\end{center}

\begin{center}
{\scriptsize
\begin{tabular}{|c|c|c|c|c|c|} \hline
$m$ & $n$ & $F_m(x,y)=\lambda$ & $\lambda$ & $F_n(x,y)=\lambda'$ & $\lambda'$\\\hline
$-18$ & $-108$ & $(0,1)$ & $-2\cdot3^4 $ & $(-18,1)$ & $ 2^2\cdot3^8$\\\hline
$-18$ & $-288$ & $(-6,1)$ & $-2\cdot3^3$ & $(-12,1)$ & $-2^5\cdot3^3\cdot5^3$\\\hline
$-27$ & $-3087$ & $(-45,1)$ & $ 3^8 $ & $(-504,1)$ & $ 3^4\cdot7^3\cdot37^3$\\\hline
$-27$ & $-3087$ & $(9,4)$ & $ 3^8 $ & $(315,4)$ & $ 3^4\cdot7^3\cdot37^3$\\\hline
$-27$ & $-3087$ & $(-36,5)$ & $-3^8 $ & $(-189,5)$ & $-3^4\cdot7^3\cdot37^3$\\\hline
$-36$ &	$-147456$ & $(3,1)$ & $3^3$ & $(528,1)$ & $2^{12}\cdot3^3\cdot71^3$\\\hline
$-38$ & $-8208$ & $(3,1)$ & $-5^3 $ & $(-126,1)$ & $2^3\cdot3^{15}$\\\hline
$-45$ & $-16875$ & $(-9,1)$ & $ 3^4 $ & $(90,1)$ & $-3^8\cdot5^3\cdot7^3$\\\hline
$-49$ & $-189$ & $(-63,1)$ & $ 7^2\cdot13^3 $ & $(-126,1)$ & $ 3^{12}\cdot7$\\\hline
$-49$ & $-189$ & $(-42,5)$ & $-7^2\cdot13^3 $ & $(-63,5)$ & $-3^{12}\cdot7$\\\hline
$-49$ & $-189$ & $(21,4)$ & $ 7^2\cdot13^3 $ & $(63,4)$ & $ 3^{12}\cdot7$\\\hline
$-54$ & $-71874$ & $(9,2)$ & $ 3^6 $ & $(693,2)$ & $ 3^9\cdot11^3\cdot13^3$\\\hline
$-54$ & $-432$ & $(-9,1)$ & $-3^6 $ & $(-18,1)$ & $-2^3\cdot3^9$\\\hline
$-68$ & $-918$ & $(6,1)$ & $ 2^2\cdot7^3 $ & $(36,1)$ & $ 2\cdot3^{12}$\\\hline
$-88$ & $-2376$ & $(-12,1)$ & $ 2^3\cdot5^3 $ & $(36,1)$ & $-2^3\cdot3^{12}$\\\hline
$-96$ & $-39366$ & $(-12,1)$ & $-2^5\cdot3^2 $ & $(-162,1)$ & $-2\cdot3^{13}\cdot7^3$\\\hline
$-108$ & $-288$ & $(9,1)$ & $ 3^8 $ & $(18,1)$ & $ 2^3\cdot3^4\cdot5^3$\\\hline
$-135$ & $-46305$ & $(9,1)$ & $ 3^6 $ & $(252,1)$ & $ 3^6\cdot7^3\cdot19^3$\\\hline
$-270$ & $-1440$ & $(-18,1)$ & $-2\cdot3^8 $ & $(-36,1)$ & $-2^5\cdot3^4\cdot7^3$\\\hline
$-363$ & $-4125$ & $(-22,1)$ & $ 5^3\cdot11^2 $ & $(55,1)$ & $-5^3\cdot11\cdot17^3$\\\hline
$-368$ & $-1058$ & $(46,3)$ & $-2^3\cdot17^3\cdot23 $ & $(-115,3)$ & $ 23^2\cdot29^3$\\\hline
$-400$ & $-106480$ & $(35,2)$ & $ 5^2\cdot11^3 $ & $(715,2)$ & $ 5\cdot11^3\cdot181^3$\\\hline
$-432$ & $-71874$ & $(18,1)$ & $-2^3\cdot3^6 $ & $(-297,1)$ & $ 3^6\cdot11^3\cdot13^3$\\\hline
$-675$ & $-3645$ & $(45,2)$ & $-3^9\cdot5^2 $ & $(-135,2)$ & $ 3^9\cdot5\cdot7^3$\\\hline
$-1404$ & $-4992$ & $(36,1)$ & $ 2^2\cdot3^{10} $ & $(72,1)$ & $ 2^7\cdot3^2\cdot17^3$\\\hline
$-1890$ & $-8400$ & $(-45,1)$ & $-3^{10}\cdot5 $ & $(-90,1)$
& $-2^3\cdot3^2\cdot5^2\cdot19^3$\\\hline
$-2784$ & $-45414$ & $(348,7)$ & $-2^5\cdot3^2\cdot23^3\cdot29 $ & $(-1566,7)$
& $2\cdot3^7\cdot29^2\cdot31^3$\\\hline
$-4913$ & $-103823$ & $(68,1)$ & $ 5^3\cdot17^3 $ & $(329,1)$
& $ 23^3\cdot47^3$\\\hline
$-14040$ & $-54080$ & $(117,1)$ & $ 3^{12}\cdot13 $ & $(234,1)$
& $2^3\cdot13^2\cdot53^3$\\\hline
$-15498$ & $-64288$ & $(-126,1)$ & $-2\cdot3^{12}\cdot7$ & $(-252,1)$
& $-2^5\cdot5^3\cdot7^2\cdot11^3$\\\hline
\end{tabular}
}
\end{center}

\vspace*{5mm}
\begin{acknowledgment}
The authors would like to thank anonymous referees for reading the original manuscript very 
carefully and for suggesting improvements and corrections. 
\end{acknowledgment}


{\small
\hspace*{-0.5cm}\\
\begin{tabular}{ll}
Akinari HOSHI & Katsuya MIYAKE\\
Department of Mathematics & Department of Mathematics\\
Rikkyo University & School of Fundamental Science and Engineering\\
3--34--1 Nishi-Ikebukuro Toshima-ku & Waseda University\\
Tokyo, 171--8501, Japan & 3--4--1 Ohkubo Shinjuku-ku\\
E-mail: \texttt{hoshi@rikkyo.ac.jp} & Tokyo, 169--8555, Japan\\
& E-mail: \texttt{miyakek@aoni.waseda.jp}
\end{tabular}

\end{document}